\documentclass[10pt, a4paper]{amsart}
\usepackage[latin1]{inputenc}

\usepackage{a4wide}
\usepackage[T1]{fontenc}
\usepackage{amsthm}
\usepackage{amsfonts}
\usepackage{graphicx,color}
\usepackage{amssymb}
\usepackage{amscd}
\usepackage{amsmath}
\usepackage{color}
\usepackage{algorithm}
\usepackage{algpseudocode}
\usepackage{enumerate}
\usepackage{hyperref}

\newtheorem{theorem}{Theorem}
\newtheorem{lemma}{Lemma}

\newtheorem{definition}{Definition}[section]

\theoremstyle{definition}
\newtheorem{example}{Example}[section]
\theoremstyle{remark}
\newtheorem{remark}{Remark}[section]
\newtheorem*{acknowledgement}{Acknowledgement}

\renewcommand{\MR}[1]{}
\newcommand{\N}{\mathbb{N}}
\newcommand{\Z}{\mathbb{Z}}
\newcommand{\wNAF}{{w\text{-NAF}}}

\newcommand{\h}{h}
\newcommand{\assign}{:=}
\newcommand{\bigOh}{\mathcal{O}}
\newcommand{\eps}{\varepsilon}
\newcommand{\Expect}{\mathbb{E}}
\newcommand{\Var}{\mathbb{V}}
\DeclareMathOperator{\Cov}{Cov}
\newcommand{\diag}{\text{diag}}
\newcommand{\val}{\mathsf{value}}
\newcommand{\wtW}{\widetilde{W}}
\newcommand{\wtX}{\widetilde{X}}
\newcommand{\wtT}{\widetilde{T}}

\newcommand{\unique}{\mathsf{unique}}
\newcommand{\nonunique}{\mathsf{nonunique}}
\newcommand{\upper}{\mathsf{upper}}

\renewcommand{\ell}{l}

\title{Analysis of the binary asymmetric joint sparse form}
\author{Clemens Heuberger}
\address{Institut f\"ur Mathematik, Alpen-Adria-Universit\"at Klagenfurt,
  Universit\"atsstra\ss e 65--67, 9020 Klagenfurt am W\"orthersee, Austria and
  Institut f\"ur Optimierung und Diskrete Mathematik (Math B), TU Graz,
  Steyrergasse 30, 8010 Graz, Austria}
\email{clemens.heuberger@aau.at}
\thanks{The authors are supported by the Austrian Science Fund (FWF):
  P 24644-N26.}
\author{Sara Kropf}
\address{Institut f\"ur Mathematik, Alpen-Adria-Universit\"at Klagenfurt,
  Universit\"atsstra\ss e 65--67, 9020 Klagenfurt am W\"orthersee, Austria}
\email{sara.kropf@aau.at}

\begin{document}

\begin{abstract}
  We consider redundant binary joint digital expansions of integer vectors. The
  redundancy is used to minimize the Hamming weight, i.e., the number of
  nonzero digit vectors. This leads to efficient linear combination algorithms
  in abelian groups, which are for instance used in elliptic curve
  cryptography.
  
  If the digit set is a set of contiguous integers containing zero, a
  special syntactical condition is known to minimize the weight. We analyze the
  optimal weight of all non-negative integer vectors with maximum entry less
  than $N$. The expectation and the variance are given with a main term and a
  periodic fluctuation in the second order term. Finally, we prove asymptotic
  normality.
\end{abstract}  
\maketitle

  \section{Introduction}

We deal with integer representations of vectors of integers called joint
representations.
\begin{definition}
  For base $2$, dimension $d$ and a digit set $D\subseteq \Z$, the
  dimension-$d$ \emph{joint representation} of a vector $n\in\Z^{d}$ is a word
  $(\eps_L\ldots \eps_0)$ with $\eps_i\in D^d$ and
  $n=\val(\eps_{L}\ldots\eps_{0})$ with $\val(\eps_{L}\ldots\eps_{0})=\sum_{i=0}^L\eps_i2^i$.
\end{definition}

Such representations can be used for computing a linear combination
$m_{1}P_{1}+\cdots+m_{d}P_{d}$ of points $P_{i}$ of an elliptic curve, or more
generally an abelian group
(cf. \cite{straus}). For every nonzero digit $\eps_{i}$, an elliptic curve
addition is performed. Since these are expensive, we want to minimize the
number of nonzero digits. On the other side, every nonzero column vector $\eps\in D^{d}$
corresponds to a precomputed point. The number of doublings corresponds to the
length of the expansion. Each $\eps_{i}$ in the expansion
$(\eps_{L}\ldots\eps_{0})$ is called a column vector of the expansion.

\begin{example}\label{ex:jointexp}
  A dimension-$3$ digit expansion with digit set
  $\{0,1,2\}$ is
$$\begin{pmatrix}1011\\0020\\2001\end{pmatrix}.$$
It is a representation  of $(11,4,17)^{T}$, because
$$\begin{pmatrix}11\\4\\17\end{pmatrix}=\val\begin{pmatrix}1011\\0020\\2001\end{pmatrix}=\begin{pmatrix}1\\0\\2\end{pmatrix}2^{3}+\begin{pmatrix}0\\0\\0\end{pmatrix}2^{2}+\begin{pmatrix}1\\2\\0\end{pmatrix}2^{1}+\begin{pmatrix}1\\0\\1\end{pmatrix}2^{0}.$$
\end{example}

\begin{definition}
  The \emph{Hamming weight} $\h(\eps_{L}\ldots\eps_{0})$ of a digit expansion
  $(\eps_{L}\ldots\eps_{0})$ is the number of nonzero columns $\eps_{i}\neq0$.
\end{definition}

\begin{example}
Continuing with Example \ref{ex:jointexp}, we have the Hamming weight $$\h\begin{pmatrix}1011\\0020\\2001\end{pmatrix}=3.$$
\end{example}

The Hamming weight of an integer depends on the representation we use. For
example, we have two representations of $4=\val(12)=\val(100)$ with Hamming weight
$\h(12)=2$ and $\h(100)=1$. But since we always use a specific digit
expansion in this paper, we just write $\h(n)$ for the Hamming weight of
this digit expansion.

This specific digit expansion is the asymmetric joint sparse form (short AJSF) as presented
by Heuberger and Muir in \cite{asymmdigits}. The AJSF is the unique
dimension-$d$ joint integer representation in base $2$ with digit set
$D_{l,u}=\{a\in\Z\mid l\leq a\leq u\}$ described in Theorem~\ref{th:AJSF} (see
\cite[Theorem 6.1]{asymmdigits}). There, Heuberger and Muir proved that the
AJSF is colexicographically minimal and has minimal Hamming weight among all
representations with this digit set $D_{l,u}$.

The width-$w$ nonadjacent form \cite{muirwNAF, avanzi} and the simple joint sparse form \cite{GHP} are special cases of the 
asymmetric joint sparse form. For the width-$w$ nonadjacent form, we use
$l=-2^{w-1}+1$, $u=2^{w-1}-1$ and dimension $1$. The simple joint sparse
form has digit set $D_{-1,1}$ and dimension $2$. The special case of Theorem~\ref{thm:asywgen}
for the simple joint sparse form has been proved in \cite{GHP}. For further results on syntactically defined optimal digit expansions, we refer to
\cite{asymmdigits} and the references therein.

We compute the expected value, the variance and the asymptotic
distribution of the Hamming weight of the AJSF. We obtain a main term plus a
periodic fluctuation and an error term, similar to the asymptotic estimates of digital sums in \cite{Flajolet-Grabner-Kirschenhofer-Prodinger:1994:mellin}. The definitions and algorithms
of the AJSF are recalled in Section~\ref{sec:pre}. In Section~\ref{sec:constr-trans}, we
construct a transducer from this algorithm. In Theorem~\ref{thm:transdimd}, we
explicitly describe this transducer to compute the Hamming weight. In
Section~\ref{sec:asymp-distr-hamm}, we prove the following Theorem~\ref{thm:asywgen} about the asymptotic normal distribution
of the Hamming weight. We use the discrete probability space $\{n\in\Z\mid 0\leq n<N\}^{d}$
with uniform distribution as a probabilistic model, in contrast to
\cite{asymmdigits}. There, only residue classes modulo powers of $2$ have been
considered in the ``full-block-length'' analysis. 

\begin{theorem}\label{thm:asywgen}
    	The Hamming weight $\h(m_1,\ldots,m_d)$ of the AJSF of an integer
        vector $(m_1,\ldots,m_d)^T$ over the digit set $D_{l,u}$ in dimension
        $d$ with equidistribution of all vectors $(m_{1},\ldots,m_{d})^{T}$
        with $0\leq m_{i}<N$ for an integer $N$ is asymptotically normally
        distributed. There exist constants $e_{\ell,u,d}$,
        $v_{\ell,u,d}\in\mathbb R$ and $\delta>0$, depending
        on $u$, $l$ and $d$, such that the expected value is
        $$e_{\ell,u,d}\log_2N+\Psi_1(\log_2N)+\mathcal O(N^{-\delta}\log N)$$
         and the variance is 
        $$v_{\ell,u,d}\log_2 N-\Psi_1^2(\log_2N)+\Psi_2(\log_2N)
    +\mathcal O(N^{-\delta}\log^2N),$$
    where $\Psi_1$ and $\Psi_2$ are continuous, $1$-periodic functions on $\mathbb R$. In particular, we have
    	$$\mathbb P\left(\frac{\h(m_1,\ldots,m_d)-e_{\ell,u,d}\log_2N}{\sqrt{v_{\ell,u,d}\log_2N}}<x\right)=\int_{-\infty}^xe^{\frac{-y^2}{2}}dy+\mathcal O\left(\frac 1{\sqrt[4]{\log N}}\right)$$
    	for all $x\in\mathbb R$.
      For $d=1$, we have
  \begin{equation*}
    e_{\ell,u,1}=\frac{1}{w-1+\lambda} \text{~~~and~~~}
    v_{\ell,u,1}=\frac{(3-\lambda)\lambda}{(w-1+\lambda)^3},
  \end{equation*}
  where 
  \begin{align*}
    \lambda&=\frac{2
              (u-\ell+1)-(-1)^\ell-(-1)^u}{2^w}
  \end{align*}
  and $w$ is the unique integer such that $2^{w-1}\leq
  u-l+1<2^{w}$. Furthermore, for $d=1$, the function $\Psi_{1}(x)$ is nowhere differentiable.
  General formulas for $e_{\ell,u,d}$ for $d=2$ are given in
  \cite[Table~3]{asymmdigits}. For $d\in\{1,2,3,4\}$,
  general formulas for $e_{\ell,u,d}$ and $v_{\ell,u,d}$ are given in
  \cite{heuberger:colexi-webpage2006}.
    \end{theorem}

    For higher dimension or the variance, 
the question of non-differentiability of the periodic fluctuations
remains open.
    
    In the last Section~\ref{sec:roots}, we further investigate the error term
    of the expected value and the variance of the Hamming weight in the case
    of the width-$w$ nonadjacent form. In this case, we have
    $\delta=\log_{2}\left(1+\frac{3\pi^{2}}{w^{3}}\right)$ for sufficiently large $w$, see Theorem~\ref{thm:wNAF}.

\section{Preliminaries}\label{sec:pre}

First, we define some properties of the digit set.

\begin{definition}
Let $D_{l,u}=\{a\in\Z\mid l\leq a\leq u\}$ for $l\leq0$ and
$u\geq1$ be the digit set. It  contains $u-l+1$ digits. We define $w$ to be the
unique integer s.t. $2^{w-1}\leq u-l+1<2^{w}$.\end{definition}

 Because
$0,1\in D_{l,u}$, we have $w\geq2$. The digit set contains
at least a complete set of residues modulo $2^{w-1}$. However, some residues
modulo $2^{w}$ are not contained. Thus we define the following sets:

\begin{definition}
Let 
\begin{align*}\unique(D_{l,u})&=\{a\in D_{l,u}\mid u-2^{w-1}<a<l+2^{w-1}\},\\
\nonunique(D_{l,u})&=D_{l,u}\setminus\unique(D_{l,u}),\\
\upper(D_{l,u})&=\{a\in D_{l,u}\mid u-2^{w-1}<a\leq u\}.\end{align*}
 The sets $\unique(D_{l,u})$ and $\nonunique(D_{l,u})$ contain
the unique respectively non-unique residues modulo $2^{w-1}$. The set
$\upper(D_{l,u})$ is a complete set of
representatives modulo $2^{w-1}$.
\end{definition}

Without loss of generality, we can restrict $l$ to be greater than
$-2^{w-1}$. Otherwise, we would take the digit set $D_{-u,-l}$ where we have
$-2^{w-1}<-u\leq-1$. Then every representation of a vector $n$ of integers
with digit set $D_{l,u}$ would correspond to a representation of $-n$ with
digit set $D_{-u,-l}$ by changing the sign of each digit. By this transformation, the weight
of the representation does not change.

\begin{theorem}[\cite{asymmdigits}]\label{th:AJSF}Let $D_{l,u}$ be a digit set and
  $n\in\Z^{d}$ (with $n\geq0$ if $l=0$). Then there exists exactly one
  representation  $(\eps_{L}\ldots\eps_{0})$ (up to leading $0$'s) of $n$,
  such that the following conditions are satisfied:
\begin{enumerate}
\item Each column $\eps_{j}$ is $0$ or contains an odd digit.
\item If $\eps_{j}\neq 0$ for some $j$, then $\eps_{j+w-2}=\cdots=\eps_{j+1}=0$.
\item If $\eps_{j}\neq 0$ and $\eps_{j+w-1}\neq0$ for some $j$, then
  \begin{enumerate}
    \item there is an $i\in\{1,\ldots,d\}$ such that $\eps_{j+w-1,i}$ is odd
      and $\eps_{j,i}\in\unique(D_{l,u})$,
    \item if $\eps_{j,i}\in\nonunique(D_{l,u})$, then
      $\eps_{j+w-1,i}\not\equiv u+1\mod2^{w-1}$,
    \item if $\eps_{j,i}\in\upper(D_{l,u})\cap\nonunique(D_{l,u})$, then
      $\eps_{j+w-1,i}\equiv u\mod 2^{w-1}$.
  \end{enumerate}
\end{enumerate}
\end{theorem}

\begin{definition} The digit expansion
  described in Theorem \ref{th:AJSF} is called \emph{asymmetric joint sparse
  form} (short AJSF) of $n$ with digit set $D_{l,u}$.
\end{definition}

\begin{example}
The AJSF of $(7,11)^{T}$ with digit set $D_{-2,3}$ is
$$\begin{pmatrix}100\bar1\\1003\end{pmatrix},$$
where $\bar1$ is the digit $-1$. Thus its Hamming weight is $\h(7,11)=2$.
\end{example}

We also consider the width-$w$ nonadjacent form (cf. \cite{muirwNAF, avanzi}).

\begin{definition}
  The \emph{width-$w$ nonadjacent form} (short $\wNAF$) of an integer $n$ is a
  radix-$2$ representation $(\eps_L\ldots\eps_0)$ of $n$ with the digit set $D_w\assign\{0,\pm1,\pm3,\ldots,\pm(2^{w-1}-3),\pm(2^{w-1}-1)\}$ and the following property:
$$\text{If }\eps_i\neq 0\text{, then }\eps_{i+1}=\ldots=\eps_{i+w-1}=0.$$
\end{definition}

The AJSF is a
generalization of the $\wNAF$. In the $1$-dimensional case, only odd
digits and $0$ are used in the AJSF due to Theorem \ref{th:AJSF}. After a nonzero digit, there are $w-1$ zeros. Thus,
for $l=-2^{w-1}+1$ and $u=2^{w-1}-1$, we obtain the $\wNAF$.

It is known that the $\wNAF$ representation exists and is unique for every integer
(cf. \cite{muirwNAF}).

In \cite{asymmdigits}, Heuberger and Muir introduce the AJSF, provide an algorithm to compute it, and prove its
minimality with respect to the Hamming weight.

\begin{theorem}[\cite{asymmdigits}]
The AJSF has minimal Hamming weight among all digit expansions of an integer
vector $n$ with digit set $D_{l,u}$. Algorithm 3 in \cite{asymmdigits}
computes the AJSF in
dimension $d$ for an integer vector $n$. 
\end{theorem}

 We present a slightly modified version of Algorithm 3 in \cite{asymmdigits} as
Algorithm~\ref{alg:d}. The modification takes into account that we are only interested in
the weight. Furthermore, those iterations of the while loop where the output is
already predetermined are skipped.

For simplicity, we write $n+a$, for a vector $n$ and an integer $a$, to denote that we add $a$ to every coordinate of the vector $n$.  

\begin{algorithm}
      \caption{Algorithm to compute the weight of the AJSF with digit set $D_{l,u}$}
      \label{alg:d}
      \begin{algorithmic}[1]
        \Require A vector of integers $n\in \Z^{d}$, integers $l\leq 0$, $u>0$, $n\geq 0$ if $l=0$
        \Ensure Weight $\h(n)$
        \State $h=0$
        \While{$n\neq 0$}
          \If{$n\equiv 0 \mod 2$}\label{alg:dluif1}
            \State $a=0$, $h=h+0$
            \State $m=\frac {n} 2$
          \Else\label{alg:dluelse1}
            \State $a=l+((n- l)\bmod 2^{w-1}$)
            \State $h=h+1$
            \State $m=\frac{n- a}{2^{w-1}}$
            \State $I_{\unique}=\{j\in\{1,2,\ldots,d\}\mid a_j\in\unique(D_{l,u})\}$
            \State $I_{\nonunique}=\{j\in\{1,2,\ldots,d\}\mid a_j\in\nonunique(D_{l,u})\}$
            \If{$m_j\equiv 0\mod 2$ for all $j\in I_{\unique}$}\label{alg:mjif}
            	\For{$j\in I_{\nonunique}$ such that $m_j$ is odd}
            		\State $a_j=a_j+2^{w-1}$
            		\State $m_j=m_j-1$
            	\EndFor
            \Else\label{alg:mjelse}
            	\For{$j\in I_{\nonunique}$ such that $m_j\equiv u+1\mod 2^{w-1}$}
            		\State $a_j=a_j+2^{w-1}$
            		\State $m_j=m_j-1$
            	\EndFor
            \EndIf
          \EndIf
          \State $n=m$
        \EndWhile
        \State\Return $h$
      \end{algorithmic}
    \end{algorithm}

The \textbf{if} branch in line~\ref{alg:dluif1} of Algorithm~\ref{alg:d}
makes the digit at the current position a zero column if possible. If this is not
possible, the \textbf{else} branch in line~\ref{alg:dluelse1} chooses the
smallest digit in each component which is congruent to the input. In the inner \textbf{if} and
\textbf{else} branches, the algorithm checks if we should change any non-unique
digits. In the \textbf{if} statement in line~\ref{alg:mjif}, we check whether
we can make the $(w-1)$-st digit after the current digit $0$. Otherwise, in the \textbf{else} statement in line~\ref{alg:mjelse}, we check 
whether we can increase the redundancy at the $(w-1)$-st digit after the
current digit by changing any non-unique digits at the current position.

In the $1$-dimensional case, we can further simplify
Algorithm~\ref{alg:d}. If $I_{\unique}\neq\emptyset$, then $I_{\nonunique}=\emptyset$. Thus the \textbf{else} branch in
line~\ref{alg:mjelse} will not be processed. Algorithm~\ref{alg:d1} is the
simplified version for the $1$-dimensional case.

 \begin{algorithm}
      \caption{Algorithm to compute the weight of the $1$-dimensional AJSF with digit set $D_{l,u}$}
      \label{alg:d1}
      \begin{algorithmic}[1]
        \Require Integers $n$, $l\leq 0$, $u>0$, $n\geq 0$ if $l=0$
        \Ensure $\h(n)$
        \State $h=0$
        \While{$n\neq 0$}
          \If{$n\equiv 0 \mod 2$} 
            \State $a=0$
            \State $h=h+0$
            \State $m=\frac n 2$
          \Else
            \State $a=l+((n-l)\bmod 2^{w-1}$)
            \State $h=h+1$
            \State $m=\frac{n-a}{2^{w-1}}$
            \If{$m\equiv 1\mod 2$ \textbf{and} $(n-l)\bmod 2^{w-1}\leq u-l-2^{w-1}$}\label{alg:1mjif}
              \State $a=a+2^{w-1}$
              \State $m=m-1$
            \EndIf
          \EndIf
          \State $n=m$
        \EndWhile
        \State\Return $h$
      \end{algorithmic}
    \end{algorithm}

\section{Construction of the transducers}\label{sec:constr-trans}
In this section, we describe the construction of the transducers for the
computation of the Hamming weight. We start with the easiest
case, the $\wNAF$. We will then modify the ideas to deal with the asymmetric case
of $D_{l,u}$-expansions in dimension $1$. We finally generalize the approach to
the $d$-dimensional $D_{l,u}$-expansions.

All transducers and automata take a (joint) binary expansion as input and read from
right to left. The output of the transducers is a sequence of $0$'s and $1$'s. Then
the computed Hamming weight is the number of $1$'s in this output.
\begin{figure}
 		  \centering
      \includegraphics{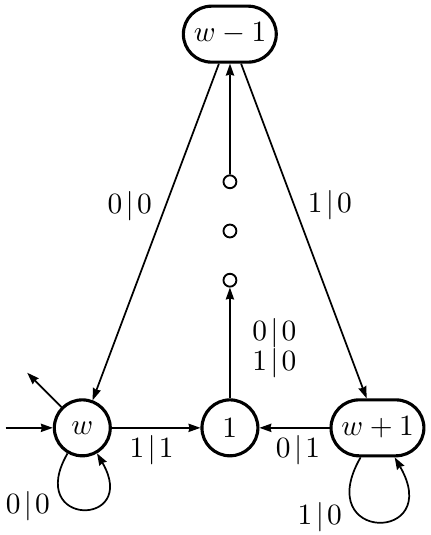}
      \caption{Transducer to compute the Hamming weight of a $\wNAF$ representation.}
      \label{aut:wNAF}
    \end{figure}
    \begin{lemma}
      Let $w\geq2$. The transducer in Figure~\ref{aut:wNAF} calculates the weight $\h(n)$ of the $\wNAF$ of an integer $n$.
    \end{lemma}
    \begin{proof}
      Let $n=\val(n_{L}\ldots n_{0})$ with $n_j\in \{0,1\}$ be the standard binary
      expansion of $n$ and $(\eps_K\ldots \eps_0)$ be the $\wNAF$
      representation of $n$. If $n\equiv 0 \mod 2$, then $\eps_0=0$ and we
      stay in the initial state. Otherwise, we have $\eps_0\neq 0$ and the
      weight is
      $\h(\eps_0)=1$. Since we have a $\wNAF$ representation, the next $w-1$
      digits fulfill $\eps_1=\eps_2=\ldots =\eps_{w-1}=0$, no matter what the
      corresponding $n_j$, $j=1,\ldots,w-1$ are. The sign of the digit
      $\eps_0$ depends on $n_{w-1} \bmod 2$. If $n_{w-1}\equiv 0\mod 2$, then
      $\eps_0 >0$ and we go to state $w$ with the next input
      $\frac{n-\eps_0}{2^w}$ with carry $0$. If $n_{w-1}\equiv 1\mod 2$, then $\eps_0<0$ and
      we therefore have a carry of $1$ and go to state $w+1$. There, reading an input of $1$ and
having a carry of $1$ results in the same outcome as reading $0$, but the carry
remains $1$. Reading an input $0$ with carry $1$ is equivalent to
reading an input $1$ with a carry $0$, so we are in state $1$ again.
    \end{proof}

In the next step, we construct a transducer for the Hamming weight of the $1$-dimensional
AJSF. Therefore, we
need the following automaton to compare integers.

\begin{lemma}
Automaton~\ref{aut:diffcomp} in Figure~\ref{aut:diffcomp} accepts the input of
three integers $a,b,c$ if and only if $a+b\leq c$. The binary expansions of
$a,b$ and $c$ must have the same length where leading $0$'s are allowed.
\end{lemma}

\begin{remark}
Automaton~\ref{aut:diffcomp} could be simplified by combining 
the input $a+b$ and merging the states $(1,0)$ and $(0,1)$. But since we use
this automaton in Theorems~\ref{th:existtrans1} and~\ref{thm:transdimd} with
the input of the transducer as $a$ and some fixed parameter as $b$, this
would complicate the constructions in Theorems~\ref{th:existtrans1} and~\ref{thm:transdimd}.
\end{remark}

\begin{proof}
The states are $(s,t)$ with $s$, $t\in\{0,1\}$. The label $s$ signifies the carry of the addition $a+b$ which still has to be processed. The label $t$ corresponds to the truth value of the expression $(a+b)\bmod2^i > c\bmod 2^i$ where $i$ is the number of read digits up to now. So the automaton accepts the input if it stops in state $(0,0)$ where there is no carry anymore and $a+b > c$ is false. The initial state is $(0,0)$. 
    
    Therefore, there is a path from $(0,0)$ to $(s,t)$ in Automaton~\ref{aut:diffcomp} with input label $$\begin{pmatrix}\alpha_{i-1}\ldots\alpha_0\\\beta_{i-1}\ldots\beta_0\\\gamma_{i-1}\ldots\gamma_0\end{pmatrix}$$
    if and only if 
    $$s=\left\lfloor\frac{\val(\alpha_{i-1}\ldots\alpha_0)+\val(\beta_{i-1}\ldots\beta_0)}{2^i}\right\rfloor$$
    and $t=\left[\left(\val(\alpha_{i-1}\ldots\alpha_0)+\val(\beta_{i-1}\ldots\beta_0)\right)\bmod2^i>\val(\gamma_{i-1}\ldots\gamma_0)\right]$. Here, we use Iverson's notation, that is $[$\emph{expression}$]$ is $1$ if \emph{expression} is true and $0$ otherwise. From this, the rules for the transitions follow. There is a transition $(s,t)\xrightarrow{(\alpha,\beta,\gamma)^T}(s',t')$ if and only if $s'=\left\lfloor\frac{\alpha+\beta+s}2\right\rfloor$ and $t'=\left[(\alpha+\beta+s)\bmod 2>\gamma-t\right]$. 
    
    \begin{figure}
      \centering
      \includegraphics{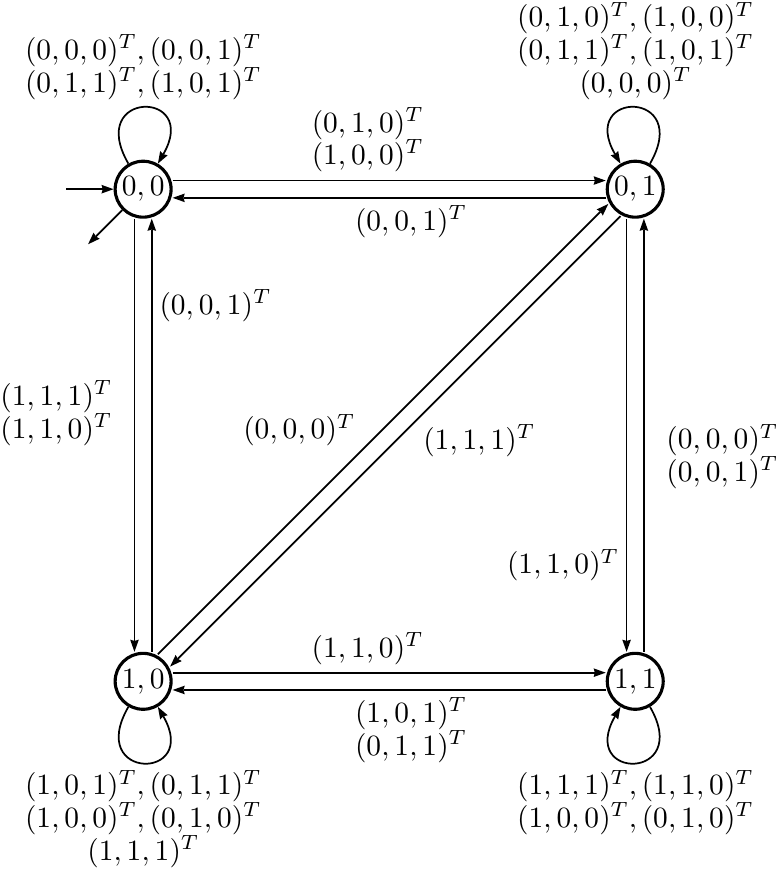}
      \caption{Automaton~\ref{aut:diffcomp} to compare three integers $a$, $b$ and $c$, accepts if $a+b\leq c$.}
      \label{aut:diffcomp}
    \end{figure}
\end{proof}

\begin{theorem}\label{th:existtrans1}
  There exists a transducer with input and output alphabet $\{0,1\}$, having less than
  $4w-2$ states, where one state is initial and final, that computes
  the Hamming weight of the AJSF from the binary expansion of an integer.
\end{theorem}
\begin{proof}
 We construct a transducer performing the same calculation as
 Algorithm~\ref{alg:d1}. It will look similar to the transducer in
 Figure~\ref{aut:wNAF}. We start at state $0$. Then there is a
 vertical block of states with $w-1$ rows having states $(0,0)_i$, $(1,0)_i$,
 $(0,1)_i$ and $(1,1)_i$ in each row $i=1,\ldots,w-1$. After this block, we either go back to state $0$, or to a similar state $1$, or again to the block of states (see Figure~\ref{aut:asydigd1F}). We call the states $0$ and $1$ the looping states. Their labels signify the carry which is to be processed. The state $0$ is also the final state.
    
    The block of states corresponds to the \textbf{if} statement in line~\ref{alg:1mjif} in Algorithm~\ref{alg:d1}. In this line, we have to check the inequality $(n-l)\bmod 2^{w-1}\leq u-l-2^{w-1}$. A first step to this aim is to compare $n+\tilde l\leq\tilde u$ with $\tilde l\assign -l$ and $\tilde u\assign u-l-2^{w-1}$. Therefore, we use Automaton~\ref{aut:diffcomp}.

		Next, we examine the binary expansions of $\tilde u$ and
                $\tilde l$. Since we have assumed that $l>-2^{w-1}$, we know
                that the length of the binary expansion of $\tilde l$ is at
                most $w-1$. Furthermore, $-1\leq \tilde u<2^{w-1}$. In the case
                $\tilde u=-1$, the set $\nonunique(D_{l,u})$ is empty and we
                have no choices for the digits. We will return to this case
                later. Then the length of the binary expansion of $\tilde u$
                is at most $w-1$. Let $(l_{w-2}\ldots l_0)$ and
                $(u_{w-2}\ldots u_0)$ be the binary expansions of $\tilde l$
                respectively $\tilde u$.
		
		Now we can verify $n+\tilde l\bmod 2^{w-1}\leq\tilde u$ by
                checking the label $t$ of the state $(s,t)$ after reading
                $w-1$ digits from the binary expansion of $(n,\tilde l,\tilde
                u)^T$ in Automaton~\ref{aut:diffcomp}. If $t=0$, then the
                inequality is true, otherwise it is false. Since the length of
                $\tilde u$ is less than or equal to $w-1$, there are no digits of $\tilde u$ left. Only a possible carry of the addition $n+\tilde l$ is left. This carry is the label $s$ of the current state $(s,t)$. Therefore, we have checked $n+\tilde l\bmod 2^{w-1}\leq \tilde u$. To ensure that we read exactly $w-1$ digits, the transducer in Figure~\ref{aut:asydigd1F} has $w-1$ copies of the four states of Automaton~\ref{aut:diffcomp}. The transitions start in a state of the $i$-th copy and go to an appropriate state of the $(i+1)$-th copy while reading the $i$-th digit of the expansion.
		
		In the \textbf{if} statement in line~\ref{alg:1mjif} in
                Algorithm~\ref{alg:d1}, we must also check the other condition
                $m\equiv 1\mod 2$. Let $(s,t)_{w-1}$ be the current state
                at the end of the block of states. We know that
                $m=\frac{(n+\tilde l)-(n+\tilde l)\bmod 2^{w-1}}{2^{w-1}}$. Therefore, the least significant digit of $m$ is simply the
                next digit of the addition $n+\tilde l$. Since there are no
                digits of the expansion of $\tilde l$ left, we only have to
                look at the next digit $\eps$ of $n$ and consider the carry $s$. Thus we have $m\equiv s+\eps\mod 2$. 
		
		If the inequality of the \textbf{if} statement is satisfied,
                that is if $t=0$, then whatever digit $\eps$ we read next, the
                transducer starts from a looping state again. If $m$ is
                even, then the next written digit is $0$ anyway. If $m$ is
                odd, we can change the digit in the representation (because it
                is non-unique) and $m$ becomes even too. We only have to
                remember the carry. If $s=0$ or $s=1$ and we read $\eps=0$,
                then there will be no carry propagation and we continue with state $0$. If $s=1$ and we read $\eps=1$, then there is a carry propagation and we start at state $1$.
		
		If the inequality is not satisfied, that is if $t=1$, and $m\equiv s+\eps\mod2$ is odd, then we have to start with the $w-1$ transitions of Automaton~\ref{aut:diffcomp} immediately. If $m$ is even however, then the transducer starts from a looping state again. In both cases, we have to consider the carry propagation as well.
		
		At state $s\in\{0,1\}$, we stay in state $s$ as long as we read
                $s$. If we read $1-s$ we start with the $w-1$ transitions of
                Automaton~\ref{aut:diffcomp}.

In the case $\tilde u=-1$, the set $\nonunique(D_{l,u})$ is empty. Therefore,
we have $t=1$ in each state, and the initial state of Automaton~\ref{aut:diffcomp} has
to be $(0,1)$. Let $(u_{w-2}\ldots u_{0})=(0^{w-1})$. Then we have a transition from $s$ to $(s',t')_{1}$ with
input label $1-s$ if and only if there is a transition from $(s,1)$ to
$(s',t')$ with input label $(1-s,l_{0},u_{0})^{T}$ in Automaton~\ref{aut:diffcomp}.
    
    To summarize, we have the following transitions in the transducer in
    Figure~\ref{aut:asydigd1F} for $s$, $s'$, $t$, $t'$, $\eps\in\{0,1\}$ and $i\in\{1,\ldots,w-2\}$:
    \begin{itemize}
    	\item $s\xrightarrow{\eps\mid0}s$ if $s=\eps$
    	\item $s\xrightarrow{\eps\mid1}(s',t')_1$ if $s\neq\eps$ and $(s,[\tilde u=-1])\xrightarrow{(\eps,l_0,u_0)^T}(s',t')$ is a transition in Automaton~\ref{aut:diffcomp}
    	\item $(s,t)_i\xrightarrow{\eps\mid0}(s',t')_{i+1}$ if $(s,t)\xrightarrow{(\eps,l_i,u_i)^T}(s',t')$ is a transition in Automaton~\ref{aut:diffcomp}
    	\item $(s,t)_{w-1}\xrightarrow{\eps\mid0}s'$ if $t=0$ or $\eps+s\equiv 0\mod 2$, and $s'=\lfloor\frac{\eps+s}2\rfloor$
    	\item $(s,t)_{w-1}\xrightarrow{\eps\mid1}(s',t')_1$ if $t=1$, $\eps+s\equiv 1\mod 2$ and $(s,[\tilde u=-1])\xrightarrow{(\eps,l_0,u_0)^T}(s',t')$ is a transition in Automaton~\ref{aut:diffcomp}.
    \end{itemize}
    
    We note that there is only one accessible state in the first row because the transitions $0\xrightarrow{1\mid1}(s,t)_1$ and $1\xrightarrow{0\mid1}(s,t)_1$ have both the same target state. This target state depends on $l$ and $u$.
    
    \begin{figure}
      \centering
      \includegraphics{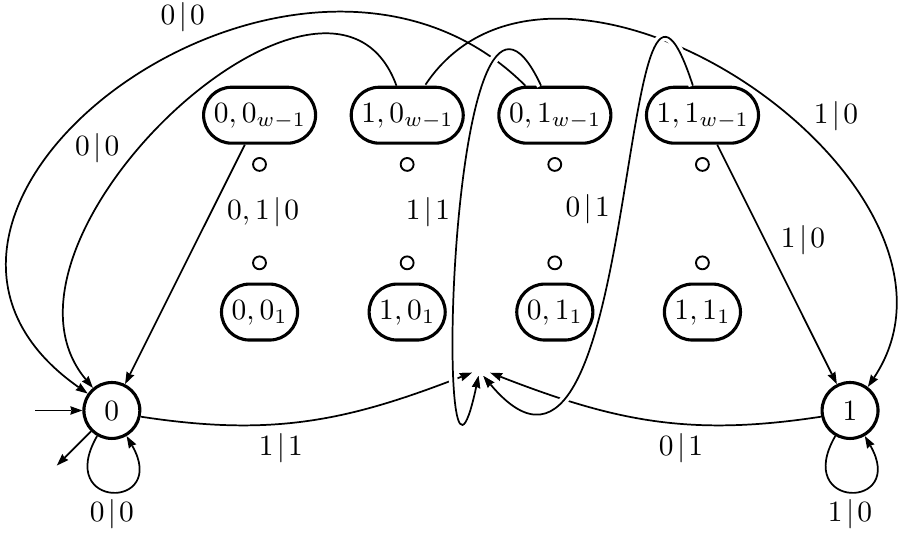}
      \caption{Transducer to compute the weight of the AJSF with digits in
        $D_{l,u}$. The transitions into and inside the block of states depend on $l$
        and $u$.}
      \label{aut:asydigd1F}
    \end{figure}

Finally, we restrict the transducer to the states which are actually
accessible from the initial state.
\end{proof}

Now we can describe the last state of the path with input label
$(\eps_{L}\ldots\eps_{0})$, a binary expansion. The following lemma can easily
be proved
by induction.

\begin{lemma}\label{lem:existtrans1}
 Let
  $k_{i}$, $s_{i}$, $t_{i}$ for $i\geq0$ and $a_{i}$, $f_{i}$ for $i\geq1$ be
  sequences with $s_{0}=0$,
  $t_{0}=1$.  The states $(s_{i},t_{i})_{w-1}$ are the states in the last row of the
 path. The integers $k_{i}$ count how often we circle in a looping state
 after the state $(s_{i}, t_{i})_{w-1}$. The integers $f_{i}$ are the positions of
 the nonzeros in the AJSF and
 $a_{i}$ is the digit at position $f_{i}$. For $\tilde u\geq0$, these
  sequences satisfy the following recursions for $i\geq1$:
\begin{align*}
  k_{i}&=\begin{cases}
    \max\{k\in\N\mid
    (\eps_{f_{i}+k+w-2}\ldots\eps_{f_{i}+w-1})=(0^{k-1}1)\text{
    or }(0^{k})\}&\text{
    if }s_{i}=t_{i}=0,\\
  \max\{k\in\N\mid
    (\eps_{f_{i}+k+w-2}\ldots\eps_{f_{i}+w-1})=(1^{k})\text{
    or }(0^{k})\}&\text{
    if }s_{i}=1,\,t_{i}=0,\\
\max\{k\in\N\mid
    (\eps_{f_{i}+k+w-2}\ldots\eps_{f_{i}+w-1})=(0^{k})\}&\text{
    if }s_{i}=0,\,t_{i}=1,\\
\max\{k\in\N\mid
    (\eps_{f_{i}+k+w-2}\ldots\eps_{f_{i}+w-1})=(1^{k})\}&\text{
    if }s_{i}=t_{i}=1,\\
    \end{cases}\\
  f_{i}&=k_{0}+\cdots+k_{i-1}+(i-1)(w-1),\\
s_{i}&=\left\lfloor\frac{\val(\eps_{f_{i}+w-2}\ldots\eps_{f_{i}+1}1)+\tilde l}{2^{w-1}}\right\rfloor,\\
t_{i}&=\left[\left(\val(\eps_{f_{i}+w-2}\ldots\eps_{f_{i}+1}1)+\tilde
    l\right)\bmod 2^{w-1}>\tilde u\right],\\
a_{i}&=-\tilde l+(\val(\eps_{f_{i+1}-1}\ldots\eps_{f_{i}+1}1)+\tilde l\bmod 2^{k_{i}+w-1}).
\end{align*}

Then we have 
$$\val(\eps_{f_{i}+w-2}\ldots\eps_{f_{i}+1}1)=a_{i}+2^{w-1}s_{i}(1-[s_{i}=1\wedge
t_{i}=0\wedge \eps_{f_{i}+w-1}=0]).$$
There is a path  from $0$ to $(s,t)_{j}$ with input label
$(\eps_{L}\ldots\eps_{0})$ if and only if $j=L-f_{i}+1$, $f_{i}\leq
L\leq f_{i}+w-2$ and
\begin{align*}
s&=\left\lfloor\frac{\val(\eps_L\ldots\eps_{f_{i}+1}1)+\val(l_{j-1}\ldots
    l_{0})}{2^{j}}\right\rfloor,\\
t&=\left[\left(\val(\eps_{L}\ldots\eps_{f_{i}+1}1)+\val(l_{j-1}\ldots
  l_{0})\right)\bmod 2^{j}>\val(u_{j-1}\ldots
u_{0})\right].
\end{align*}

There is a path from $0$ to $s$ with input label $(\eps_{L}\ldots\eps_{0})$ if
and only if $f_{i}+w-1\leq L\leq
f_{i+1}-1$ and $s=s_{i}$ or $s=0, s_{i}=1$ and
$(\eps_{L}\ldots\eps_{f_{i}+w-1})=(0^{L-f_{i}-w+2})$.

For $\tilde u=-1$ the only difference is $t_{i}=1$ and $t=1$.
\end{lemma}

    \begin{example}\label{ex:asymmdig}
    	For $l=-3$ and $u=11$, we have $w=4$, $\tilde l=(011)_2$ and $\tilde u=u-l-2^{w-1}=(110)_2$. The transducer can be seen in Figure~\ref{aut:exasymmdig}, where all non-accessible states are gray.
    \end{example}
    \begin{figure}
      \centering
      \includegraphics{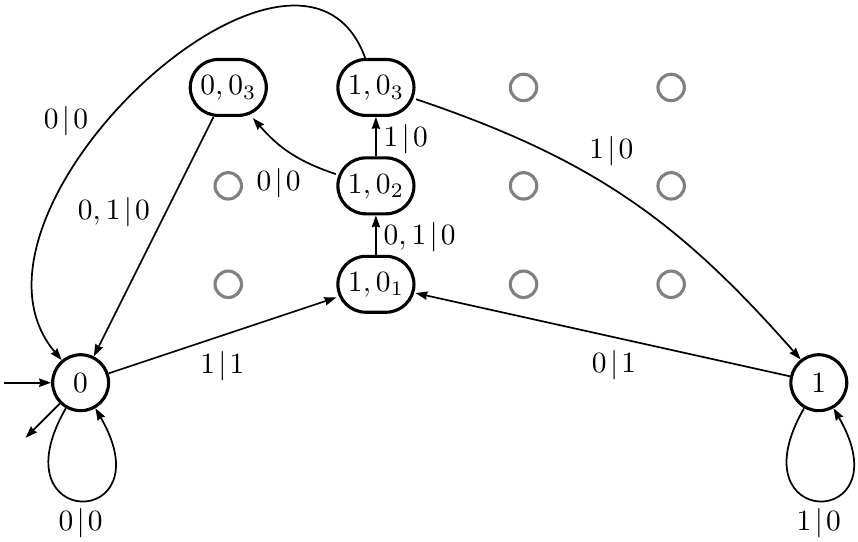}
      \caption{Transducer to compute the weight of the AJSF with digits in
        $D_{-3,11}$. The non-accessible states are gray.}
      \label{aut:exasymmdig}
    \end{figure}

We recall that a reset sequence of a transducer is a sequence $(n_{L}\ldots n_{0})$ such
that there exists a state $s$ with the following property: 
For all states $t$, if the transducer is in
state $t$ and the next input is $(n_{L}\ldots n_{0})$, then the transducer
 is in state $s$.

Now we generalize this transducer to arbitrary dimension $d$.

\begin{theorem}\label{thm:transdimd}
  There exists a transducer to compute the Hamming weight of the
  AJSF for the joint binary expansion
  of a $d$-dimensional vector of integers as input. It has one state which is
  initial and final, input and output alphabet $\{0,1\}$ and less than $8^dw$
  states.

The word $0^{4w}$ is a reset sequence of this transducer. It leads to the initial and final state
of the transducer.
\end{theorem}

\begin{proof}

 We construct a transducer calculating the weight of AJSF. In order to explain the structure of this transducer, we first consider a provisional transducer implementing a simpler version of the Algorithm~\ref{alg:d} which omits the else branch in line~\ref{alg:mjelse}, see also the algorithm on page 306 of \cite{asymmdigits}. The resulting provisional transducer is similar to the transducer in Figure~\ref{aut:asydigd1F}.
    
    For every vector
    $s\in\{0,1\}^d$, there is a state. These states are called looping states. The
    vector $s$ signifies the carry at each coordinate. The state
    $(0,\ldots,0)^T$ is the initial state. Furthermore, there is a
    block of states. The states inside the block have the labels $(s,t)_i$
    where $s$, $t\in\{0,1\}^d$, and $i$ is the row in the block. The
    coordinates of $s$ and $t$ have the same meaning as in the proof of
    Theorem~\ref{th:existtrans1}, that is: $s$ is the carry of the addition $n+\tilde l$ and $t$ signifies whether the digit is in $\nonunique(D_{l,u})$ or not.
    
    If $s\in\{0,1\}^d$ is a looping state, then there is a loop with label
    $s\mid0$ at this state. Because if we read $\eps=s$, then we have
    $\eps+s\equiv0\mod 2$, the output is $0$ and the carry propagates to the
    next step. If we read $\eps\neq s$, then we start with the $w-1$ transitions of Automaton~\ref{aut:diffcomp} in Figure~\ref{aut:diffcomp} in each coordinate. These $w-1$ transitions are processed independently for every coordinate. Therefore, we need $4^d$ states in each row and $w-1$ rows to process exactly $w-1$ transitions of Automaton~\ref{aut:diffcomp}.
    
    After the last row of the block of states, we either go back to a looping state
    or again start with the block of states immediately. Let $(s,t)_{w-1}$ be the current
    state in the last row and $\eps$ the next input digit. As in the $1$-dimensional case we have $m\equiv \eps+s\mod 2$. If for every coordinate
    $j$, $t_j=1$ implies that $m_j$ is even, then we have to process the
    \textbf{if} branch in line~\ref{alg:mjif}. We write this condition as the scalar product
    $t\cdot (s+\eps\bmod 2)=0$. In this case, the next output digit will be
    $0$ and we go on to a looping state $s'$ where the new carry is $s'=\lfloor\frac{s+\eps}2\rfloor$.
    
    If  $t\cdot(s+\eps\bmod 2)=0$ does not hold, then we would have to process the \textbf{else} branch in line~\ref{alg:mjelse}. But since we skip this part for now, we simply have to restart the transducer with the input $m$ in the case $t\cdot(s+\eps\bmod 2)>0$. We know that $m$ is the original next input $\eps$ plus the carry $s$. In this case, $s\neq\eps$, otherwise $t\cdot(s+\eps\bmod 2)>0$ would be false. Therefore, there is a transition $s\xrightarrow{\eps\mid1}(s',t')_1$ in this transducer. This ensures that, when restarting the transducer with input $m$, we immediately go on to the state $(s',t')_{1}$. Hence, we have a transition $(s,t)_{w-1}\xrightarrow{\eps\mid1}(s',t')_{1}$ in the provisional transducer.
    
    Altogether, for $s$, $s'$, $t$, $t'\in\{0,1\}^d$, $i\in\{1,\ldots w-2\}$, $j\in\{1\ldots d\}$ and $\eps\in\{0,1\}^d$, we have the following transitions in this provisional transducer:
    
    \begin{itemize}
    	\item $s\xrightarrow{\eps\mid0}s$ if $\eps=s$
    	\item $s\xrightarrow{\eps\mid1}(s',t')_1$ if $\eps\neq s$ and $\forall j:\, (s_j,[\tilde u=-1])\xrightarrow{(\eps_j,l_0,u_0)^T}(s_j',t_j')$ is a transition in Automaton~\ref{aut:diffcomp}
    	\item $(s,t)_i\xrightarrow{\eps\mid0}(s',t')_{i+1}$ if $\forall j:\,(s_j,t_j)\xrightarrow{(\eps_j,l_i,u_i)^T}(s_j',t_j')$ is a transition in Automaton~\ref{aut:diffcomp}
    	\item $(s,t)_{w-1}\xrightarrow{\eps\mid0}s'$ if $t\cdot(s+\eps\bmod 2)=0$ and $s'=\left\lfloor\frac{s+\eps}2\right\rfloor$
    	\item $(s,t)_{w-1}\xrightarrow{\eps\mid1}(s',t')_1$ if $t\cdot(s+\eps\bmod 2)>0$ and $s\xrightarrow{\eps\mid1}(s',t')_1$ is a transition in this transducer.
    \end{itemize}
    
    This transducer does the same as Algorithm~\ref{alg:d} without the
    \textbf{else} branch in line~\ref{alg:mjelse}. In the case $\tilde u=-1$
    we are finished because in the \textbf{else} branch nothing is
    done. Otherwise we must consider the \textbf{else} statement. 
    
    Let $(s,t)_{w-1}$ be the current state in the last row and $\eps$ be the next input digit. To process the \textbf{else} branch, $t\cdot(s+\eps\bmod 2)>0$ must hold in the state $(s,t)_{w-1}$. Otherwise, we would process the \textbf{if} branch. First let us examine one coordinate $j$. If $t_j=1$, nothing is done in the \textbf{else} branch because the digit at this coordinate is unique. If $t_j=0$, we have to decide whether $m_j\equiv u+1\mod2^{w-1}$. Here, $m_j\bmod 2^{w-1}$ corresponds to the next $w-1$ input digits plus the carry $s_j$ from the current state $(s,t)_{w-1}$. So we just have to compare the input letters plus the carry with the binary expansion of $u+1\bmod2^{w-1}$ or, equivalently, we compare $m_j-l\bmod 2^{w-1}$ with $\tilde v=u-l+1\bmod 2^{w-1}$. If they are not the same at some point, then we just go on like we did in the provisional transducer. 
    
    If they are the same, we have to process the \textbf{else} branch. There we would have taken $m_j-1$ as the next input of the algorithm instead of $m_j$. Therefore we have to decide where we would be in the provisional transducer, when starting in $(s,t)_{w-1}$ and the input is the original input minus $1$. This case only happens if originally the next nonzero digit is unique, but changing the current digit ensures that the next nonzero digit is non-unique. Nevertheless, the next digit will not be $0$, since this is the case when the \textbf{if} branch is processed. Therefore we would start in $(s,t)_{w-1}$ with original input minus $1$ and immediately go to the block of states again. Otherwise, the next digit would be $0$. Thus after $w-1$ transitions, we are again in a state $(s',t')_{w-1}$ in the last row. Since the next digit is non-unique, we have $t_j'=0$.
    
    To determine the value of $s'_j$, we have to decide whether there is a carry at position $w-1$ in the addition of $m_j-1$ and $\tilde l$. We have $m_j-1\bmod 2^{w-1}=u+k2^{w-1}$ for $k\in\Z$. Since $0\leq u\leq 2^{w}-1$, we have $k\in\{0,-1\}$. Then the carry is
    \begin{equation*}
    		s'_j=\left\lfloor\frac{(m_j-1)\bmod 2^{w-1}+\tilde l\bmod 2^{w-1}}{2^{w-1}}\right\rfloor
    		=\left\lfloor\frac{u+\tilde l+k2^{w-1}}{2^{w-1}}\right\rfloor
    		=1+k,
    \end{equation*}
    because $2^{w-1}\leq u+\tilde l<2^w$. Therefore, we have
    $$s_j'=\begin{cases}0&\text{ if }u\geq 2^{w-1},\\1&\text{ if }u<2^{w-1}.\end{cases}$$
    As a result, the state $(s',t')_{w-1}$ where we would be in the provisional transducer has 
    $$(\left[u<2^{w-1}\right],0)$$
     in the $j$-th coordinate.
    
    To remember that we can change the $j$-th coordinate at the end of the
    block, we have to use a second identical block $\{j\}$. Let $\emptyset$ be
    the first block, which already exists in the provisional transducer. Let
    $(s,t)_i^C$ be a state in block $C$. At the end of block $\emptyset$, we
    go to block  $\{j\}$ if $t\cdot(s+\eps\bmod 2)>0$ and $t_j=0$.  Otherwise,
    we go to a looping state or to the block $\emptyset$. If we find out that $m_j\not\equiv u+1\mod 2^{w-1}$ in block $\{j\}$, then we go back to the appropriate state in block $\emptyset$. At the end of block $\{j\}$ in the state $(s,t)_{w-1}^{\{j\}}$, we go to the same states as we would go from the state with $(\left[u<2^{w-1}\right],0)^\emptyset_{w-1}$ in the $j$-th coordinate.
    
    Up to now, we only considered one coordinate. Now we combine this approach
    for all coordinates. Since for each coordinate, we have to remember whether we are allowed to change it or not, we need one block for every subset of coordinates. Let block $C\subseteq\{1,\ldots,d\}$ be the block where we can change the coordinates in $C$. The states in block $C$ are denoted by $(s,t)_i^C$. The block $\emptyset$ is the block which already exists in the provisional transducer. The block $\{1,\ldots,d\}$ is not accessible since we need at least one unique coordinate and only non-unique coordinates can be changed.
    
    If we are in block $C\neq\emptyset$ and we find out that not every coordinate $j\in C$ satisfies $m_j\equiv u+1\mod 2^{w-1}$, we go to the appropriate state in block $C'=C\setminus\{j\in\{1,\ldots,d\}\mid m_j\not\equiv u+1\mod 2^{w-1}\}$. At the end of block $C$ in state $(s,t)_{w-1}^C$, we can change the coordinates in $C$ and all other coordinates remain the same. Therefore, we go to the same states as we would go from $(\tilde s,\tilde t)_{w-1}^\emptyset$ where $\tilde s_j=\left[u<2^{w-1}\right]$, $\tilde t_j=0$ for $j\in C$ and all other coordinates stay the same, that is $\tilde s_j=s_j$ and $\tilde t_j=t_j$ for $j\not\in C$.
    
    Let $(v_{w-2}\ldots v_0)$ be the binary expansion of $\tilde v$. Further
    let $s$, $s'$, $t$, $t'\in\{0,1\}^d$, $C$, $C'\subsetneq\{1,\ldots,d\}$, $j\in\{1\ldots d\}$, $i\in\{1,\ldots,w-2\}$ and $\eps\in\{0,1\}^d$. Then altogether there are the following transitions in the final transducer:
    
    \begin{itemize}
    	\item $s\xrightarrow{\eps\mid0}s$ if $s=\eps$
    	\item $s\xrightarrow{\eps\mid1}(s',t')_1^\emptyset$ if $s\xrightarrow{\eps\mid1}(s',t')_1$ is a transition in the provisional transducer
    	\item $(s,t)_i^C\xrightarrow{\eps\mid0}(s',t')_{i+1}^{C'}$ if $(s,t)_i\xrightarrow{\eps\mid0}(s',t')_{i+1}$ is a transition in the provisional transducer and $C'=C\setminus\{j:\,s_j+\eps_j+l_i\bmod 2\neq v_i\}$
    	\item $(s,t)_{w-1}^\emptyset\xrightarrow{\eps\mid0}s'$ if $t\cdot(s+\eps\bmod2)=0$ and $s'=\left\lfloor\frac{s+\eps}2\right\rfloor$
    	\item $(s,t)_{w-1}^\emptyset\xrightarrow{\eps\mid1}(s',t')_1^{C'}$ if $t\cdot(s+\eps\bmod 2)>0$, $(s,t)_{w-1}\xrightarrow{\eps\mid1}(s',t')_1$ is a transition in the provisional transducer and $C'=\{j:\,s_j+\eps_j+l_0\bmod 2=v_0\text{ and }t_j=0\}$
    	\item $(s,t)_{w-1}^C\xrightarrow{\eps\mid1}(s',t')_1^{C'}$ if
          $C\neq\emptyset$ and
          $(\tilde{s},\tilde{t})_{w-1}^\emptyset\xrightarrow{\eps\mid1}(s',t')_1^{C'}$
          is a transition in this transducer with
          $\tilde{s}_j=\left[u<2^{w-1}\right]$, $\tilde{t}_j=0$ for $j\in C$
          and $\tilde s_j=s_j$, $\tilde t_j=t_j$ for $j\not\in C$
          \item  $(s,t)_{w-1}^C\xrightarrow{\eps\mid0}s'$ if
          $C\neq\emptyset$ and
          $(\tilde{s},\tilde{t})_{w-1}^\emptyset\xrightarrow{\eps\mid0}s'$
          is a transition in this transducer with
          $\tilde{s}_j=\left[u<2^{w-1}\right]$, $\tilde{t}_j=0$ for $j\in C$
          and $\tilde s_j=s_j$, $\tilde t_j=t_j$ for $j\not\in C$.
    \end{itemize}

Finally, we restrict the transducer to the states which are actually
accessible from the initial state.

Due to the construction of the transducer, the sequence $0^{4w}$ leads to the
initial and final state from any state. 

It is possible to define similar sequences to those in Lemma
\ref{lem:existtrans1}, but since this requires more than one page, we omit this here.
\end{proof}

   \begin{example}\label{ex:allgdim}
    In Figure~\ref{aut:exdg}, there is a sketch of the transducer computing the
    weight of the AJSF over $D_{-2,3}$ in dimension $2$. The labels of transitions are omitted in the figure and the transitions going back at the end of a block or inside a block are gray. We have $w=3$, $\tilde u=(01)$, $\tilde l=(10)$ and $\tilde v=(10)$. 
    
    For example, the state $\begin{pmatrix}01\\11\end{pmatrix}_2^{\{2\}}$ has transitions to the same states as the state $\begin{pmatrix}01\\10\end{pmatrix}_2^\emptyset$ since $u<2^{w-1}$.
    
    \begin{figure}
      \centering
      \includegraphics{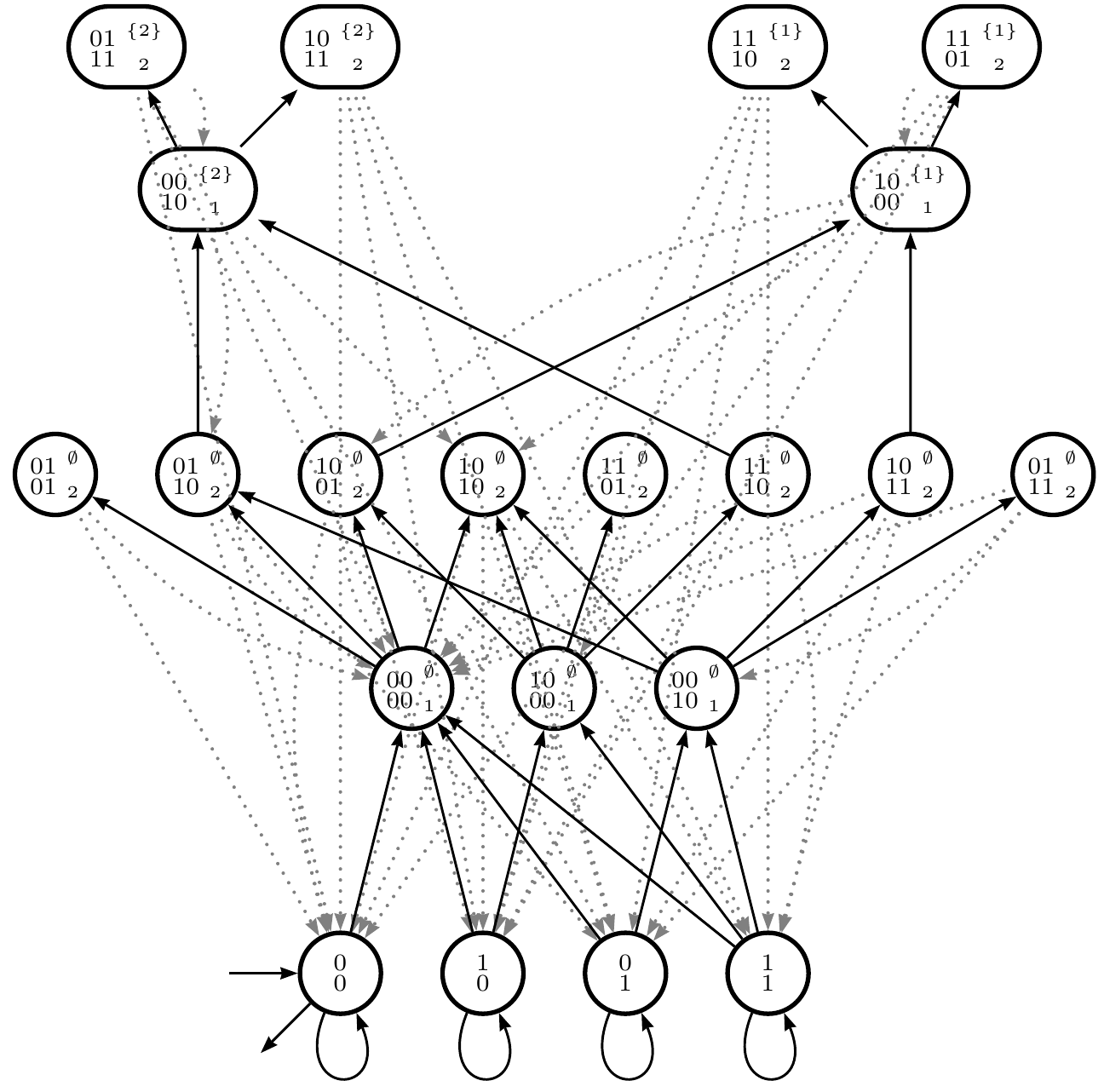}
      \caption{Transducer to compute the Hamming weight of the AJSF in dimension~$2$ over the digit set $D_{-2,3}$.}
      \label{aut:exdg}
    \end{figure}
    \end{example}

\section{Proof of Theorem \ref{thm:asywgen}}\label{sec:asymp-distr-hamm}

This section contains the proof of Theorem~\ref{thm:asywgen} which is a generalization of Theorem 6 in \cite{GHP}. With the transducer in Theorem \ref{thm:transdimd}, we can compute the
asymptotic Hamming weight. Therefore, we use the following lemma which can be proved by
induction on $L$.
	\begin{lemma}\label{lem:rec}
Let $A_0$, $A_1$ be matrices in $\mathbb{C}^{n\times n}$, $H:\mathbb{N}\to
\mathbb{C}^{n\times n}$ be any function and $G:\mathbb{N}\to\mathbb{C}^{n\times n}$ be a
function which satisfies the recurrence relation
\[ G(2N+\varepsilon)=A_\varepsilon G(N) + \varepsilon H(N) \]
for $N\ge 1$ and $\varepsilon\in\{0,1\}$. Then
		$$G\left(\sum_{p=0}^L\eps_p2^p\right)=\sum_{p=0}^L\eps_p\left(\prod_{j=0}^{p-1}A_{\eps_j}\right)H\left(\sum_{j=p+1}^L\eps_j2^{j-p-1}\right)$$
where we additionally set $H(0)=G(1)$.
	\end{lemma}
   We define
   $f(m_1,\ldots,m_d)\assign e^{it\h(m_1,\ldots,m_d)}$. The matrices
   $M_{\eps_1,\ldots,\eps_d}$ for $\eps_i\in\{0,1\}$ are defined as follows:
   The $(j,k)$-th entry of the matrix $M_{\eps_1,\ldots,\eps_d}$ is $e^{ith}$
   if there is a transition from state $j$ to $k$ with input label
   $(\eps_1,\ldots,\eps_d)^T$ and output label $h$. The entry is $0$ if there
   is no transition from state $j$ to $k$ with this input label. The ordering
   of states is
   considered to be fixed in such a way that the initial state is the last state. Then we have
    \begin{align}\label{eq:1}f(m_1,\ldots,m_d)=  v^T\prod_{p=0}^LM_{m_{1,p},\ldots,m_{d,p}}M_{0,\ldots,0}^{4w} v\end{align}
    for $ v^T=(0,\ldots,0,1)$ and $m_i=\sum_{p=0}^Lm_{i,p}2^p$. The product describes all possible paths from any state to any other state, using
edges with input labels corresponding to the input $(m_{1},\ldots,m_{d})$. The exponent of the entries of the
matrix product is the sum of output labels on these paths. Since we are interested in paths
starting and ending in state $(0,\ldots,0)^{T}$, we multiply by $v^{T}$ from the left and $v$ from the right. The factor
    $M_{0,\ldots,0}^{4w}$ is due to the reset sequence from
    Theorem~\ref{thm:transdimd} and ensures that we stop at the final state.
    
    We further define the following summatory functions
    \begin{align*}
    	E(N)&=\sum_{0\leq m_1,\ldots,m_d<N}f(m_1,\ldots,m_d),\\
    	F(N)&=\sum_{0\leq m_1,\ldots,m_d<N}M(m_1,\ldots,m_d),
    \end{align*}
    with 
    $$M(m_1,\ldots,m_d)=\prod_{p=0}^LM_{m_{1,p},\ldots,m_{d,p}}.$$
    In other words, this last equation says that the function $M(m_{1},\ldots,m_{d})$ is $2$-multiplicative
    (cf. \cite{distbinomials}).
 By~(\ref{eq:1}), we have
$$E(N)=v^{T}F(N)M_{0,\ldots,0}^{4w}v.$$

    To write down a recursion formula for $F(N)$, we need the following matrices
    $$B_{C,D}\assign\sum_{\substack{\eps_i=0,1\\i\not\in C\cup D}}\sum_{\substack{\eps_i=0\\i\in C}}\sum_{\substack{\eps_i=1\\i\in D}}M_{\eps_1,\ldots,\eps_d}$$
    for disjoint $C$, $D\subseteq\{1,\ldots,d\}$. The first index $C$ of $B_{C,D}$ is
    the set of coordinates where the digit is $0$. The second index $D$ is
    the set of coordinates where the digit is $1$. All other coordinates in
    $(C\cup D)^c$ can be any digit. By
construction, we have $\|M_{\varepsilon_1, \ldots, \varepsilon_d}\|_1=1$, where
$\|\,\cdots\,\|_1$ denotes the row sum norm of a matrix. We conclude that
$\|B_{C,D}\|_1\le 2^{d-|C|-|D|}$.
 As a special matrix we define $A=B_{\emptyset,\emptyset}$.
    
    Furthermore we define the functions
    $$G_C(N)\assign\sum_{\substack{0\leq m_i<N\\i\not\in
        C}}\sum_{\substack{m_i=N\\i\in C}}M(m_1,\ldots,m_d)$$
    for every set $C\subseteq \{1,\ldots,d\}$.
    
    Then we have $F(N)=G_\emptyset(N)$, and the functions satisfy the following
    recursion formulas due to $2$-multiplicativity
    \begingroup\allowdisplaybreaks
    \begin{align*}
    	G_C(2N)&=\sum_{\substack{\eps_i=0,1\\i\not\in C}}\sum_{\substack{\eps_i=0\\i\in C}}\sum_{\substack{2m_i+\eps_i<2N\\i\not\in C}}\sum_{\substack{2m_i+\eps_i=2N\\i\in C}}M(2m_1+\eps_1,\ldots,2m_d+\eps_d)\\
    	&=B_{C,\emptyset}G_C(N),\\
    	G_C(2N+1)&=\sum_{\substack{\eps_i=0,1\\i\not\in C}}\sum_{\substack{\eps_i=1\\i\in C}}\sum_{\substack{2m_i+\eps_i<2N+1\\i\not\in C}}\sum_{\substack{2m_i+\eps_i=2N+1\\i\in C}}M(2m_1+\eps_1,\ldots,2m_d+\eps_d)\\
    	&=\sum_{D\subseteq C^c}\sum_{\substack{\eps_i=0,1\\i\not\in C\cup D}}\sum_{\substack{\eps_i=0\\i\in D}}\sum_{\substack{\eps_i=1\\i\in C}}M_{\eps_1,\ldots,\eps_d}\sum_{\substack{m_i<N\\i\not\in C\cup D}}\sum_{\substack{m_i=N\\i\in C\cup D}}M(m_1,\ldots,m_d)\\
    	&=\sum_{D\subseteq C^c}B_{D,C}G_{C\cup D}(N).
    \end{align*}\endgroup

	From this recursion, we can determine $G_C(N)$ inductively because all
        functions $G_{C'}$ required for computing $G_{C}$ have $C'\supsetneq
        C$. Therefore, we have the following recursion formula for $F(N)=G_{\emptyset}(N)$
	$$F(2N+\eps)=AF(N)+\eps H(N)$$
	for $N\geq 1$, $\eps\in\{0,1\}$ and 
	$$H(N)=\sum_{\emptyset\neq D\subseteq \{1,\ldots,d\}}B_{D,\emptyset}G_{D}(N).$$

		If we define $H(0)=G_{\emptyset}(1)$, we can use Lemma~\ref{lem:rec} and get
		\begin{align}\label{eq:fd}F\left(\sum_{p=0}^L\eps_p2^p\right)=\sum_{p=0}^L\eps_pA^pH\left(\sum_{j=p+1}^L\eps_j2^{j-p-1}\right).\end{align}
		Here, $H(N)$ is considered to be a known function because it is a sum of functions $G_C(N)$, which are recursively known by Lemma~\ref{lem:rec}.
		
		From the definition of $G_C(N)$, we can derive the growth rates of the functions $G_C(N)$ and $H(N)$. We have $\|G_C(N)\|_{1}=\mathcal O(N^{d-\vert C\vert})$ and $\|H(N)\|_{1}=\mathcal O(N^{d-1})$.
    
		Next, we investigate the eigenvalues of the matrix
                $A$. We first consider the case $t=0$. In this case,
                $A$ is the adjacency matrix of the
                underlying graph of the transducer
                 in Theorem  \ref{thm:transdimd}. Therefore, it has an
                eigenvalue $2^d$ with eigenvector $(1,\ldots,1)^T$. By the
                theorem of Perron-Frobenius, there is a unique dominant
                eigenvalue $\mu(0)$ of $A$ which is easily seen to be primitive as
                every state is reachable from any other state in exactly $4w$
                steps. As $\|A\|_1\leq 2^d$ and the largest eigenvalue is always at most $\|A\|_1$, $\mu(0)=2^d$ is the largest eigenvalue. We denote the modulus of the second largest eigenvalue
                by $\beta(0)$. Since
                eigenvalues are continuous, for $t$ in a suitable neighborhood
                of $0$, $A$ has a unique dominant eigenvalue $\mu(t)$ and the
modulus
                $\beta(t)$ of the second largest eigenvalue fulfills
                $\beta(t)<|\mu(t)|$.
    
    Now we want to split up (\ref{eq:fd}) into two parts, one for the
    dominating eigenvalue and one for the remaining eigenvalues. Therefore, let
    $J=T^{-1}AT$ be a Jordan decomposition of $A$ where $J$ has been sorted
    such that it has $\mu(t)$ in the upper left corner. We define $\Lambda\assign T\diag(\mu(t)^{-1},0,\ldots,0)T^{-1}$ and $R\assign T(J-\diag(\mu(t),0,\ldots,0))T^{-1}$. Then $A^p=\mu^L\Lambda^{L-p}+R^p$ holds for $p\leq L$. Further, we define
    \begin{align}
    	\Lambda(x_0,x_1,\ldots)&=\sum_{p=0}^{\infty}x_p\Lambda^pH\left(\sum_{j=0}^{p-1}x_j2^{p-1-j}\right),\label{eq:lambda}\\
    	R\left(\sum_{p=0}^{L}\eps_p2^{p}\right)&=\sum_{p=0}^{L}\eps_pR^pH\left(\sum_{j=p+1}^L\eps_j2^{j-p-1}\right).\nonumber
    \end{align}
    The function $\Lambda$ is well defined on the infinite product space
    $\{0,1\}^\N$ because it is dominated by a geometric series. We extend $\Lambda$ to a function on
    $[1,2)$ by setting
    $\Lambda\left(\sum_{p=0}^{\infty}x_{p}2^{-p}\right)\assign
    \Lambda(x_{0},x_{1},\ldots)$ with the standard  binary expansion and choosing
    the representation ending on $0^\omega$ in the case of ambiguity.
    
    Then we have
    $$F\left(\sum_{p=0}^L\eps_p2^p\right)=\mu^L\Lambda(\eps_L,\eps_{L-1},\ldots,\eps_0,0^\omega)+R(\eps_L,\ldots,\eps_0)$$
    and
    \begin{align*}E(N)=\mu(t)^{\log_2 N}\Psi(\log_2N,t)+ \tilde R(N,t)\end{align*}
    with $\Psi(x,t)=\mu(t)^{-\{x\}}
    v^T\Lambda(2^{\{x\}})M_{0,\ldots,0}^{4w} v$ and $\tilde
    R(N,t)=v^{T}R(N)M_{0,\ldots,0}^{4w}v$. Furthermore, there is a $\delta\in(0,1]$
    such that $\log_{2}\beta(t)<d-\delta$ in a suitable
    neighborhood of $0$. Then we have
    $$\vert\tilde R(N,t)\vert=\mathcal O(N^{d-\delta}).$$

    So we have
    \begin{align}\label{eq:expsum}E(N)=N^{d+a_1t+a_2t^2+\mathcal
        O(t^3)}\Psi(\log_2N,t)+\mathcal O(N^{d-\delta})\end{align}
with $a_{1}$ and $a_{2}$ depending on the Taylor expansion of $\log_{2}\mu(t)$
at $t=0$. If we insert $t=0$ in~(\ref{eq:expsum}), we obtain
$\psi_{0}=\Psi(\log_{2}N,0)=1+\bigOh\left(N^{-\delta}\right)$.

       The function $\Psi(x,t)$ is periodic in $x$ with period $1$ and is well
       defined for all $x\in\mathbb R^+$. To prove continuity in $x$, we first note that continuity for
       $x\in[0,1)$ with $x=\log_2y$ where $y$ is not a dyadic rational
       follows from (\ref{eq:lambda}). To prove it for $x=\log_2y$ with
       $y=\sum_{p=0}^L\eps_p2^{-p}$ a dyadic rational with $\eps_L=1$, we observe that the two
       one-sided limits exist due to (\ref{eq:lambda}). Next, we prove that they are the same. Therefore, we look at the two sequences $N_k=y2^{L+k}$ and $\tilde N_k=y2^{L+k}-1$. Then
    $$\lim_{k\rightarrow\infty}2^{\{\log_2N_k\}}=(\eps_0\centerdot\eps_1\ldots\eps_{L-1}10^\omega)\text{ and }\lim_{k\rightarrow\infty}2^{\{\log_2\tilde N_k\}}=(\eps_0\centerdot\eps_1\ldots\eps_{L-1}01^\omega).$$
    If we insert these two sequences in (\ref{eq:expsum}), we get
    $$\mathcal O(N_{k}^{d-1})=E(N_k)-E(\tilde N_k)=N_k^{d}\Psi(\log_2N_k,t)-\tilde N_k^{d}\Psi(\log_2\tilde N_k,t)+\mathcal O(N_k^{d-\delta}),$$
    and hence $\lim_{k\rightarrow\infty}\Psi(\{\log_2
    N_k\},t)=\lim_{k\rightarrow\infty}\Psi(\{\log_2\tilde N_k\},t)$. Therefore,
    $\Psi(x,t)$ is continuous in $x$.
       
     In $t$, $\Psi(x,t)$ is also continuous because the eigenvalues of a
     matrix are continuous. Furthermore, the function $\Psi(x,t)$ is
     arbitrarily often differentiable in $t$ because it is dominated by a
     geometric series. By the same argument as above, these derivatives are
     continuous in $x$.
    
    The first and second derivative of $E(N)$ with respect to $t$ at $t=0$ imply that the expected value of the Hamming weight is
    \begin{align}\label{eq:2}\frac{1}{N^d}\sum_{m_i<N}\h(m_1,\ldots,m_d)=e_{\ell,u,d}\log_2N+\Psi_1(\log_2N)+\mathcal O(N^{-\delta}\log N)\end{align}
    with $e_{\ell,u,d}=-ia_1\log2$ and $\Psi_1(\log_2N)=-i\frac{\partial}{\partial t}\Psi(\log_2 N,t)\vert_{t=0}$, and
    \begin{align*}\frac1{N^d}\sum_{m_i<N}\h^2(m_1,\ldots,m_d)&=v_{\ell,u,d}\log_2N+e_{\ell,u,d}^2\log^2_2N+2e_{\ell,u,d}\log_2N\Psi_1(\log_2N)\\&\qquad+\Psi_2(\log_2N)+\mathcal O(N^{-\delta}\log^2N)\nonumber\end{align*}
    with $v_{\ell,u,d}=-2a_2\log 2$ and $\Psi_2(\log_2N)=-\frac{\partial^2}{\partial t^2}\Psi(\log_2 N,t)\vert_{t=0}$. From that, we calculate the variance which is
    \begin{align*}
    \frac1{N^d}\sum_{m_i<N}\h^2(m_1,\ldots,m_d)&-\left(\frac1{N^d}\sum_{m_i<N}\h(m_1,\ldots,m_d)\right)^2=\\& v_{\ell,u,d}\log_2 N-\Psi_1^2(\log_2N)+\Psi_2(\log_2N)
    +\mathcal O(N^{-\delta}\log^2N).
    \end{align*}
    We first compute the characteristic function $\hat g_N(t)$ of the random variable 
    	$$Z=\frac{\h(m_1,\ldots,m_d)-e_{\ell,u,d}\log_2 N}{\sqrt{v_{\ell,u,d}\log_2 N}},$$
    	which is
    	\begin{align*}
    		\hat g_N(t)&=\frac 1{N^d}\sum_{m_{i}<N}e^{it\frac{\h(m_1,\ldots,m_d)-e_{\ell,u,d}\log_2N}{\sqrt{v_{\ell,u,d}\log_2N}}}\\
    		&=e^{\frac{-t^2}2}\left(1+\mathcal O\left(\frac{t^3}{\log^{3/2} N}\right)\right)\psi\left(\log_2N,\frac t{\sqrt{v_{\ell,u,d}\log_2N}}\right)\\&\quad+\frac1{N^{d}}\tilde R\left(N,\frac t{\sqrt{v_{\ell,u,d}\log_2N}}\right)e^{-it\frac{e_{\ell,u,d}}{\sqrt{v_{\ell,u,d}}}\sqrt{\log_2N}}.
    	\end{align*}
    	
    Since $\hat g_{N}(t)$ is a characteristic function,  we have
    $1=\psi_{0}+r_{0}$ for $r_{0}=\frac1{N^{d}}\tilde R(N,0)$ and $\psi_{0}=\Psi(\log_{2}
    N,0)$. We know that $\frac1{N^{d}}\tilde R(N,t)=\mathcal O(N^{-\delta})$. Next, we can estimate the difference from $\hat g_N(t)$ to the characteristic function $\hat f(t)=e^{-\frac{t^2}2}$ of the normal distribution with mean $0$ and variance $1$, which is
    	\begin{align*}\vert \hat g_N(t)-\hat f(t)\vert&=\left\vert e^{-\frac{t^2}2}\left(1+\mathcal O\left(\frac{t^3}{\log^{3/2}N}\right)\right)\left(\psi_0+\mathcal O\left(\frac t{\sqrt{\log N}}\right)\right)\right.\nonumber\\
      &\quad+\left.\frac1{N^{d}}\tilde R\left(N,\frac t{\sqrt{v_{\ell,u,d}\log_2N}}\right)\exp\left(-it\frac{e_{\ell,u,d}}{\sqrt{v_{\ell,u,d}}}\sqrt{\log_2N}\right)-(\psi_0+r_0)e^{-\frac{t^2}2}\right\vert\nonumber\\
  &=\mathcal O\left(\frac t{\sqrt{\log N}}\right),          
\end{align*}
for $t=o(\sqrt{\log N})$.
    	
    	Therefore, the Berry-Esseen inequality (cf.~\cite{vaaler}) implies
    	\begin{align*}
    	\mathbb{ P}\left(\frac{\h(m_1,\ldots,m_d)-e_{\ell,u,d}\log_2N}{\sqrt{v_{\ell,u,d}\log_2N}}<x\right)=\frac1{\sqrt{2\pi}}\int_{-\infty}^xe^{-\frac{y^2}2}dy+\mathcal O\left(\frac 1{\sqrt[4]{\log N}}\right).
    \end{align*}

 For a specific digit set and dimension, we can compute the constants $e_{\ell,u,d}$ and $v_{\ell,u,d}$ explicitly.

\begin{example}\label{ex:adjmat}We consider the digit set $D_{-2,3}$ in dimension $2$. See
  Example~\ref{ex:allgdim} and Figure~\ref{aut:exdg} for the transducer. The
  adjacency matrix $A$ of the underlying graph of this transducer is given in
  Table~\ref{tab:allgdim}, where $z=e^{it}$.
    \begin{table}$$A=\begin{pmatrix}
    0& 0& 0& 0& 0& 1& 1& 1& 1& 0& 0& 0& 0& 0& 0& 0& 0& 0& 0& 0& 0\\
    0& 0& 0& 0& 0& 0& 0& 1& 1& 1& 1& 0& 0& 0& 0& 0& 0& 0& 0& 0& 0\\
    0& 0& 0& 0& 0& 0& 1& 0& 1& 0& 0& 1& 1& 0& 0& 0& 0& 0& 0& 0& 0\\
   0& 0& 0& 0&  0& 0& 0& 1& 1& 0& 0& 0& 0& 1& 1& 0& 0& 0& 0& 0& 0\\
   0& 0& 0& 0& 0&  0& 1& 0& 1& 0& 0& 0& 0& 0& 0& 1& 1& 0& 0& 0& 0\\
   3 z& 0& 0& 0& 0&  0& 0& 0& 0& 0& 0& 0& 0& 0& 0& 0& 0& 0& 0& 0& 1\\
  z& 0& 0& 0& z& 0&   0& 0& 0& 0& 0& 0& 0& 0& 0& 0& 0& 0& 1& 0& 1\\
  z& 0& 0& z& 0& 0& 0&   0& 0& 0& 0& 0& 0& 0& 0& 0& 0& 0& 0& 1& 1\\
  0& 0& 0& 0& 0& 0& 0& 0&   0& 0& 0& 0& 0& 0& 0& 0& 0& 1& 1& 1& 1\\
  2 z& z& 0& 0& 0& 0& 0& 0&   0& 0& 0& 0& 0& 0& 0& 0& 0& 0& 0& 1& 0\\
  z& 0& 0& 0& z& 0& 0& 0& 0&   0& 0& 0& 0& 0& 0& 0& 0& 1& 0& 1& 0\\
  z& 0& 0& z& 0& 0& 0& 0& 0& 0&   0& 0& 0& 0& 0& 0& 0& 1& 1& 0& 0\\
  2 z& 0& z& 0& 0& 0& 0& 0& 0& 0&   0& 0& 0& 0& 0& 0& 0& 0& 1& 0& 0\\
  0& 0& 0& 0& 0& 0& 0& 0& 0& 0& 0&   0& 0& 0& 0& 0& 0& 1& 1& 1& 1\\
  z& 0& 0& z& 0& 0& 0& 0& 0& 0& 0& 0&   0& 0& 0& 0& 0& 0& 0& 1& 1\\
  z& 0& 0& 0& z& 0& 0& 0& 0& 0& 0& 0& 0&   0& 0& 0& 0& 0& 1& 0& 1\\
  0& 0& 0& 0& 0& 0& 0& 0& 0& 0& 0& 0& 0& 0&   0& 0& 0& 1& 1& 1& 1\\
  z& z& z& 0& 0& 0& 0& 0& 0& 0& 0& 0& 0& 0& 0&   0& 0& 1& 0& 0& 0\\
  2 z& 0& z& 0& 0& 0& 0& 0& 0& 0& 0& 0& 0& 0& 0&   0& 0& 0& 1& 0& 0\\
  2 z& z& 0& 0& 0& 0& 0& 0& 0& 0& 0& 0& 0& 0& 0&   0& 0& 0& 0& 1& 0\\
  3 z& 0& 0& 0& 0& 0& 0& 0& 0& 0& 0& 0& 0& 0& 0&   0& 0& 0& 0& 0& 1 
    \end{pmatrix}$$\caption{Adjacency matrix of the transducer in Example~\ref{ex:adjmat}}\label{tab:allgdim}\end{table}

    The characteristic polynomial of $A$ is
    $$-(x-1) x^7 \left(x^2-2 z\right) \left(x^3-x^2-xz-2 z\right)^2 \left(x^5-x^4-7 x^3z-20 x^2 z+6 x z^2-24 z^2\right).$$
    At $t=0$, the dominating eigenvalue $\mu(0)=4$ is a root of the fourth factor. Therefore the Taylor expansion of $\mu(t)$ around $t=0$ is
    $$\mu(t)=4+\frac{128i}{89}t-\frac{673216}{2114907}t^2+\mathcal O(t^3).$$
    
    Hence the expected value of the Hamming weight is
    $$\frac{32}{89}\log_2 N+\mathcal O(1)$$
    and the variance is
    $$\frac{63200}{2114907}\log_2N+\mathcal O(1).$$
\end{example}
\bigskip

In order to determine the constants $e_{\ell,u,d}$ and $v_{\ell,u,d}$  giving
mean and variance in general,
we rephrase the results of the ``full-block-length'' analysis in
\cite{asymmdigits} in a probability model which is easily compared with our
main results.

\begin{lemma}\label{lm:dluvert2}
  Let $k\ge w$ be a positive integer. Let $\wtW_k$ be the Hamming weight of the
  AJSF of a random vector $m=(m_1,\ldots,m_d)^T$  with
  equidistribution of all vectors $m=(m_1,\ldots,m_d)^T$ with $0\le m_i<2^k$.

  Then 
  \[
    \Expect \wtW_k=e_{\ell,u,d} k+ \bigOh(1) \text{~~~and~~~}
    \Var \wtW_k=v_{\ell,u,d} k+ \bigOh(\sqrt k)
    \]
  for the constants given in Theorem~\ref{thm:asywgen}.
\end{lemma}
\begin{proof}
  For $j<k$, we denote the $j$-th digit of
  the AJSF of a random vector $m=(m_1,\ldots,m_d)^T$ by $\wtX_j$, where we
  assume equidistribution of all vectors $m=(m_1,\ldots,m_d)^T$ with $0\le
  m_i<2^k$. 

  In \cite[\S~6.2]{asymmdigits}, the random variables $X_j$ denoting the $j$-th
  digit of a random AJSF has been considered, where the probability measure was
  defined to be the image of the Haar measure on the space of $d$-tuples  of
  $2$-adic integers under the AJSF, i.e., equidistribution on all residue
  classes modulo $2^\ell$ for all $\ell$ has been assumed. Furthermore, $W_j$
  was defined to be the weight of the first $j$ digits.

  From Algorithm~\ref{alg:d}, it is clear that $\wtX_j$ only depends
  on $m$ modulo $2^{j+w}$. This implies that $\wtX_j$ and $X_j$ are identically
  distributed for $j<k-w$. Therefore $W_{k-w}$ and $\wtW_{k-w}$ are identically distributed, too. Furthermore, we always have $|\wtW_k-\wtW_{k-w}|\le w$. 

  By \cite[Theorem~6.7]{asymmdigits}, we have 
  \[
    \Expect \wtW_{k-w}=\Expect W_{k-w} =e_{\ell,u,d} (k-w)+ \bigOh(1) \text{~~~and~~~}
    \Var \wtW_{k-w}=\Var W_{k-w} =v_{\ell,u,d} (k-w)+ \bigOh(1).
    \]
  
  We conclude that
  \begin{align*}
     \Expect \wtW_k &=\Expect \wtW_{k-w} +\bigOh(1)=\Expect W_{k-w}+\bigOh(1)=e_{\ell,u,d}k+\bigOh(1),\\
     \Var \wtW_k &=\Var \wtW_{k-w} +\Var(\wtW_{k}-\wtW_{k-w})+2\Cov(\wtW_{k-w},
    \wtW_k-\wtW_{k-w})=\Var W_{k-w}+\bigOh(\sqrt{k}),
  \end{align*}
  where the Cauchy-Schwarz inequality has been used in the form
  \[ \Cov(\wtW_{k-w},
  \wtW_k-\wtW_{k-w}) \le \sqrt{\Var \wtW_{k-w}\Var(\wtW_k-\wtW_{k-w})}. \]
\end{proof}

In the next lemma, we prove that the function $\Psi_{1}(x)$
is non-differentiable at any real number in the $1$-dimensional case. The
proof uses the method presented by Tenenbaum~\cite{Tenenbaum:1997:non-derivabilite}, see also Grabner and Thuswaldner~\cite{thuswaldner}. 

\begin{lemma}
Let $d=1$. Then the function $\Psi_{1}(x)$  in Theorem
\ref{thm:asywgen} is nowhere differentiable.
\end{lemma}
\begin{proof}
Let $g(N)=\left\lfloor 2^{-1-4w}N^{\frac{c}{c+1}}\right\rfloor$ be a positive integer valued
function with $c\in\mathbb Z$ and $c>\frac{1}{\delta}-1$. We have $g(N)=o(N)$ and $N^{1-\delta}\log
N=o(g(N))$.

Assume $\Psi_{1}$ is differentiable at $x\in[0,1)$. Let
$2^{x}=\sum_{p=0}^{\infty}\eps_{p}2^{-p}$ be the standard binary digit expansion
choosing the representation ending on $0^{\omega}$ in the case of
ambiguity. Further, let $x_{k}$, $y_{k}$ and $N_{k}$ be such that
$2^{x_{k}}=\sum_{p=0}^{k}\eps_{p}2^{-p}$, $N_{k}=2^{k(c+1)+x_{k}}\in\mathbb
Z$ and $2^{k(c+1)+y_{k}}=N_{k}+g(N_{k})$. Then we have
\begin{align*}
x-x_{k}&=\mathcal O(2^{-k}),\\
y_{k}-x_{k}&=\log_{2}\left(1+\frac{g(N_{k})}{N_{k}}\right)=\frac{1}{\log 2}\frac{g(N_{k})}{N_{k}}+\mathcal
O\left(\frac{g(N_{k})^{2}}{N_{k}^{2}}\right),\\
\lim_{k\rightarrow\infty}y_{k}&=x.
\end{align*}

Because of the choice of $g(N)$, we
have 
\begin{align*}g(N_{k})&<2^{ck-4w},\\
\frac{g(N_{k})}{N_{k}}&=\Theta(2^{-k}).\end{align*}

We have $\h(2^{p+4w}n+m)=\h(n)+\h(m)$ for
$p\geq0$ and $m<2^{p}$ because $0^{4w}$ is a reset sequence leading to the
initial state (see Theorem~\ref{thm:transdimd}). Due to~(\ref{eq:2}) and the periodicity and continuity of $\Psi_{1}$, we have
\begin{align}\label{eq:4}\sum_{N_{k}\leq n<N_{k}+g(N_{k})}\h(n)&=g(N_{k})\h(N_{k})+\sum_{n<g(N_{k})}\h(n)\\
&=g(N_{k})\h(N_{k})+e_{l,u,1}g(N_{k})\log_{2}g(N_{k})+g(N_{k})\Psi_{1}\left(\frac{c}{c+1}x\right)+o(g(N_{k})).\nonumber
\end{align}

On the other hand, we have

\begin{align}\label{eq:5}\sum_{N_{k}\leq
    n<N+g(N_{k})}\h(n)&=e_{l,u,1}(N_{k}+g(N_{k}))\log_{2}(N_{k}+g(N_{k}))+(N_{k}+g(N_{k}))\Psi_{1}(\log_{2}(N_{k}+g(N_{k})))\\&\qquad-e_{l,u,1}N_{k}\log_{2}N_{k}-N_{k}\Psi_{1}(\log_{2}N_{k})+O(N_{k}^{1-\delta}\log
N_{k})\nonumber\\
&=e_{l,u,1}N_{k}(y_{k}-x_{k})+N_{k}(\Psi_{1}(y_{k})-\Psi_{1}(x_{k}))+e_{l,u,1}g(N_{k})(k(c+1)+y_{k})\nonumber\\
&\qquad+g(N_{k})\Psi_{1}(y_{k})+o(g(N_{k})).\nonumber
\end{align}

If we divide by $g(N_{k})$ in (\ref{eq:4}) and (\ref{eq:5}), then we obtain
\begin{align*}
\h(N_{k})+e_{l,u,1}\log_{2}g(N_{k})+\Psi_{1}\left(\frac{c}{c+1}x\right)&=e_{l,u,1}(y_{k}-x_{k})\frac{N_{k}}{g(N_{k})}+\frac{N_{k}}{g(N_{k})}(\Psi_{1}(y_{k})
-\Psi_{1}(x_{k}))\\&\qquad+e_{l,u,1}(k(c+1)+y_{k})+\Psi_{1}(y_{k})+o(1).
\end{align*}

Now we can write the difference of the $\Psi_{1}$ on the right-hand side in
terms of the derivative 
$$\Psi_{1}(y_{k})-\Psi_{1}(x_{k})=\Psi_{1}'(x)(y_{k}-x_{k})+o(x-x_{k})+o(\vert
y_{k}-x\vert),$$
and we get
\begin{align*}
\h(N_{k})&=-\Psi_{1}\left(\frac{c}{c+1}x\right)+\frac{e_{l,u,1}}{\log
2}+\frac{\Psi_{1}'(x)}{\log 2}+e_{l,u,1}(k+\frac{x}{c+1})+\Psi_{1}(x)+o(1).
\end{align*}
Next, we take the difference of two subsequent terms
\begin{align}\label{eq:6}\h(N_{k+1})-\h(N_{k})=e_{l,u,1}+o(1),
\end{align}
where the left-hand side is an integer. We have $e_{l,u,1}\not\in\mathbb Z$ because
$0<e_{l,u,1}=\frac{1}{w-1+\lambda}<1$.

Therefore the
right-hand side of~(\ref{eq:6}) is not an integer for $k$ large enough. This contradicts our
assumption that $\Psi_{1}$ is differentiable in $x$.
\end{proof}

\section{Asymptotic distribution of the \texorpdfstring{$\wNAF$}{w-NAF}}\label{sec:roots}

In this section, we specialize the result of Theorem~\ref{thm:asywgen} to the $\wNAF$.

\begin{theorem}\label{thm:wNAF}
  The weight $\h(n)$ of the $\wNAF$ of the integer $n$ with equidistribution
  on $\{n\in\Z\mid 0\leq n<N\}$ is asymptotically normally distributed. There
  exists a $\delta>0$ such that the mean is 
  $$\frac1{w+1}\log_2N+\Psi_1(\log_2N)+\mathcal O(N^{-\delta}\log N)$$
   and the variance is
   $$\frac2{\left(w+1\right)^{3}}\log_2 N-\Psi_1^2(\log_2N)+\Psi_2(\log_2N)
    +\mathcal O(N^{-\delta}\log^2N),$$
    where $\Psi_1$ and $\Psi_2$ are continuous, $1$-periodic functions on $\mathbb
    R$. If $w$ is large enough, then $\delta=\log_{2}\left(1+\frac{3\pi^{2}}{w^{3}}\right)$. In particular, we have
    $$\mathbb P \left(\frac{\h(n)-\frac{\log_2N}{w+1}}{\sqrt{\frac2{\left(w+1\right)^3}\log_2N}}<x\right)=\frac1{\sqrt{2\pi}}\int_{-\infty}^xe^{-\frac{y^2}2}dy+\mathcal O\left(\frac1{\sqrt[4]{\log N}}\right)$$
    for all $x\in\mathbb R$.
\end{theorem}

This follows from the following Lemma \ref{lem:charpoly} and Theorem \ref{thm:asywgen}.

\begin{lemma}\label{lem:charpoly}
The characteristic polynomial of the matrix $A$ of the transducer in
Figure~\ref{aut:wNAF} is $(x-1)(x^{w}-x^{w-1}-2^{w-1}e^{it})$. The largest
eigenvalue $\mu(t)$ is unique around $t=0$ and $\mu(0)=2$. Furthermore, for
large enough $w$ and $t=0$, there is exactly one simple eigenvalue in $T_k=\{x\in \mathbb C:|x|\geq \frac2{1+\frac3w},\,|\arg
x-\frac{2k\pi}w|\leq\frac{\pi}{2w}\}$ for each $k=0,\ldots,w-1$. Additionally
there is the obvious 
eigenvalue $x=1$. The eigenvalues with the second largest absolute value are in
$T_{1}$ and $T_{w-1}$. For each eigenvalue at $t=0$, an expansion in $\frac1w$
can be computed with arbitrarily small error term.
\end{lemma}

\begin{proof}
The characteristic polynomial of $A$ is obtained by Laplace expansion. With
$z=\frac2x$, the interesting factor of the characteristic polynomial is transformed into $z^{w}+z-2=0$. The smallest root in absolute value of this polynomial is $1$
because for $|z|<1$ we have 
$$|z^{w}+z|=|z|\cdot |z^{w-1}+1|<2|z|<2.$$
 We use the fixed point equation $f_{k}(z)=z$
with $$f_{k}(z)=(2-z)^{\frac1w}e^{\frac{2\pi ik}{w}}$$ for
$k=0,\ldots,w-1$. Here, we take the main branch of the $w$-th root. After the substitution, we have $\wtT_{k}=\{z\in\mathbb C:\vert
z\vert\leq1+\frac3w,\,|\arg z-\frac{2k\pi}{w}|\leq\frac{\pi}{2w}\}$, which corresponds to $T_{-k\bmod w}$. For
$w$ large enough and $|z|\leq1+\frac3w$, we have
\begin{align*}\vert f_{k}'(z)\vert&=\frac1w\vert2-z\vert^{\frac1w-1}\\
&\leq\frac1w\left(1-\frac3w\right)^{\frac1w-1}\leq\frac1{w-3}
\end{align*}
and
\begin{align*}
\vert f_{k}(z)\vert&\leq\vert f_{k}(z)-f_{k}(1)\vert+\vert f_{k}(1)\vert\\
&\leq \frac1{w-3}\vert z-1\vert+1\leq\frac{2+\frac3w}{w-3}+1\leq 1+\frac3w.
\end{align*}
Furthermore, we have $|\arg f_{k}(z)-\frac{2k\pi}w|\leq\frac{\pi}{2w}$. Thus, $f_{k}(\wtT_{k})\subseteq
\wtT_{k}$, and $f_{k}$ is a contraction on $\wtT_{k}$ with
Lipschitz constant $(w-3)^{-1}<1$. Therefore, there exists a unique fixed point of $f_{k}$
in $\wtT_{k}$ for each $k$. Because $\wtT_{k}$ for $k=0,\ldots,w-1$ only intersect
in $0$, which is certainly no root of the polynomial, we found $w$ distinct
roots of the polynomial $z^{w}+z-2$. Thus, we found all roots of this polynomial.  We only have to investigate $0\leq k\leq \frac
w2$ because the coefficients of the polynomial are real. Let $z\in \wtT_{k}$ be the fixed point of
$f_{k}$. For $\rho=\exp(\frac{2\pi i}w)$, we have
$|z-\rho^{k}|=\bigOh(\frac1w)$.  Therefore, $z=\rho^{k}+\bigOh(\frac 1w)$.

For $k\leq w^{\alpha}$ with a fixed
$\alpha\in\left(\frac12,1\right)$, we have
\begin{align*}
z&=f(z)=1+\frac{2\pi ik}w+\bigOh\left(\frac {k^2}{w^2}\right).
\end{align*}
Iterating, we successively get
\begin{align*}
z&=1+\frac{2\pi
  ik}w-\frac{2\pi^{2}k^{2}}{w^{2}}-\frac{2\pi ik}{w^2}+\bigOh\left(\frac{k^3}{w^3}\right)\text{\quad
and}\\
|z|&=1+\frac{4\pi^{2}k^{2}}{w^{3}}+\bigOh\left(\frac{k^3}{w^4}\right).
\end{align*}
Therefore, for large $w$, only the
fixed points for $k=w-1,0,1$ are in the disk $\{z\in\mathbb C:|z|\leq1+\frac{10\pi^{2}}{w^{3}}\}$.

For $k\geq w^{\alpha}$,
we have
\begin{align*}
|2-\rho^{k}|&=\left(5-4\cos\left(\frac{2k\pi}w\right)\right)^{\frac12}\geq\left(5-4\cos\left(2w^{\alpha-1}\pi\right)\right)^{\frac12}\\
&=1+4\pi^{2}w^{2\alpha-2}+\bigOh(w^{4\alpha-4}),\\
|2-z|&\geq|2-\rho^{k}|-|z-\rho^{k}|\\
&=1+4\pi^{2}w^{2\alpha-2}+\bigOh(w^{4\alpha-4}+w^{-1}),\\
|f_{k}(z)|&=\exp\left(\frac1{w}\log|2-z|\right)\geq 1+4\pi^{2}w^{2\alpha-3}+\bigOh(w^{4\alpha-5}+w^{-2}).
\end{align*}
Thus for $k\geq w^{\alpha}$, the fixed point of $f_{k}$ is not in the disk
$\{z\in\mathbb C:|z|\leq1+\frac{10\pi^{2}}{w^{3}}\}$ for large $w$.
\end{proof}

\begin{acknowledgement}
  We thank the anonymous referees for their constructive comments and for
  encouraging us to prove the non-differentiability of $\Psi_1$.
\end{acknowledgement}
  \bibliographystyle{amsplain}
  \bibliography{lit}

\end{document}